\documentclass[11pt]{amsart}
\usepackage[T1]{fontenc}
\usepackage{lmodern}
\usepackage{amsmath,amssymb,amsthm,mathtools}
\usepackage{microtype}
\usepackage{indentfirst}
\usepackage{needspace}
\usepackage[colorlinks=true,linkcolor=blue,citecolor=blue,urlcolor=blue]{hyperref}
\numberwithin{equation}{section}
\newtheorem{theorem}{Theorem}[section]
\newtheorem*{theorem*}{Theorem}
\newtheorem{proposition}[theorem]{Proposition}
\newtheorem*{proposition*}{Proposition}
\newtheorem{lemma}[theorem]{Lemma}

\theoremstyle{remark}
\newtheorem{remark}[theorem]{Remark}
\newtheorem{example}[theorem]{Example}
\newcommand{\Ric}{\operatorname{Ric}}
\newcommand{\Scal}{\operatorname{Scal}}
\newcommand{\Vol}{\operatorname{Vol}}
\newcommand{\Area}{\operatorname{Area}}
\newcommand{\Hess}{\operatorname{Hess}}
\newcommand{\diver}{\operatorname{div}}
\newcommand{\tr}{\operatorname{tr}}
\newcommand{\dd}{\,d}
\newcommand{\R}{\mathbb R}
\newcommand{\Sph}{\mathbb S}
\newcommand{\Om}[2]{\Omega_{#1,#2}}
\title[Scalar curvature growth in dimension three]{Scalar curvature growth on nonnegatively curved three-manifolds}
\author{Qixuan Hu}
\address{Department of Mathematics and Computer Sciences,
Shantou University, Guangdong, P. R. China}
\email{qxhu@stu.edu.cn}
\thanks{This work was supported by SRIG under Grant NTF25026T}
\date{\today}

\begin{document}

\begin{abstract}
Let $(M^3,g)$ be a complete, connected, noncompact Riemannian
three-manifold without boundary and with nonnegative sectional
curvature. We prove that its scalar-curvature integral over geodesic
balls, divided by the radius, has a limit. In the one-ended case,
\[
 \lim_{r\to\infty}\frac1r\int_{B_p(r)}\Scal\dd V
 =8\pi\bigl(\chi(M)-V_M\bigr)\le8\pi(1-V_M),
\]
where $V_M$ is the asymptotic volume ratio. In the one-ended case,
we show that $\chi(M)\in\{0,1\}$.
Thus positive asymptotic volume ratio gives the value
$8\pi(1-V_M)$, while in the collapsed case the value is determined
by the topology of the end. No pole or scalar-curvature bound is
assumed. The proof combines separate smooth approximations of the
Busemann and distance functions, integrable negative curvature errors,
and a determinant estimate on the level surfaces. An averaged boundary
estimate then gives convergence of the integrated extrinsic curvature.
In the two-ended case, the splitting theorem gives the exact limit
$8\pi\chi(N)$ for the compact surface factor $N$.
The one-ended upper bound is attained for every prescribed asymptotic
volume ratio in $[0,1]$, and the two-ended bound is also sharp.
\end{abstract}

\maketitle

\section{Introduction}\label{sec:introduction}

The Cohn--Vossen inequality \cite{CohnVossen} bounds the total
Gaussian curvature of a complete, oriented, noncompact surface with
nonnegative curvature by $2\pi\chi(M)$. Seeking a higher-dimensional
analogue, Yau \cite[Problem~9]{YauProblems} asked about the growth
of integrals of the elementary symmetric functions of the Ricci
tensor on complete manifolds with nonnegative Ricci curvature.
The scalar-curvature case asks whether every complete noncompact
$(M^n,g)$ with $\Ric\ge0$ satisfies
\begin{equation}\label{eq:yau-question}
 \limsup_{r\to\infty}r^{2-n}
 \int_{B_p(r)}\Scal\dd V<\infty.
\end{equation}
In dimension three, the normalization is division by the radius.

Early progress included the estimates of Shi and Yau \cite{ShiYau}
in the K\"ahler setting under boundedness, pinching, and nonnegative
holomorphic bisectional curvature assumptions. Yang \cite{Yang}
constructed counterexamples for higher elementary symmetric functions
of the Ricci tensor. The scalar-curvature question itself has recently
been answered negatively: Hao and Zhu \cite{HaoZhu} and, independently,
Xu \cite{XuCollapse} constructed complete metrics on $\R^3$ with
strictly positive Ricci curvature and unbounded normalized
scalar-curvature integrals. Their examples have zero asymptotic volume
ratio. Xu also constructed compact examples that rule out a dimensional
unit-ball bound under $\Ric\ge0$. In dimensions $n\ge4$, Cheng
\cite{Cheng} obtained counterexamples even with a pole. These results
make the distinction between Ricci and sectional curvature essential.

Petrunin \cite{Petrunin} proved a dimensional local bound under a
lower sectional-curvature bound. Under $K\ge0$, scaling gives a bound
for the normalized integral at every radius. The present paper
determines its asymptotic value in dimension three, including the
contribution of the topology of the end.

Throughout the paper, $(M^3,g)$ is smooth, complete, connected,
noncompact, and without boundary. Write $S=\Scal$ and, for a fixed
$p\in M$, set
\[
 \rho(x)=d(p,x),\qquad I_p(r)=\int_{B_p(r)}S\dd V,\qquad
 V_M=\lim_{r\to\infty}\frac{\Vol B_p(r)}{(4\pi/3)r^3}.
\]
Under $\Ric\ge0$, Bishop--Gromov comparison gives the existence and
basepoint independence of $V_M$, with $0\le V_M\le1$.
When $K\ge0$, the soul theorem \cite{CheegerGromollStructure} gives
$M$ finite homotopy type, so its Euler characteristic $\chi(M)$ is
well defined.

\begin{theorem}\label{thm:main}
Let $(M^3,g)$ be a complete, connected, noncompact Riemannian
three-manifold without boundary, with $K\ge0$. Then $M$ has either
one or two ends, and the following conclusions hold for every $p\in M$.
\begin{enumerate}
\renewcommand{\theenumi}{\roman{enumi}}
\renewcommand{\labelenumi}{\textup{(\theenumi)}}
\item\label{case:one} If $M$ has one end, then $\chi(M)\in\{0,1\}$ and
\begin{equation}\label{eq:main}
 \lim_{r\to\infty}\frac{I_p(r)}r
 =8\pi\bigl(\chi(M)-V_M\bigr)\le8\pi(1-V_M).
\end{equation}
\item\label{case:two} If $M$ has two ends, then it is isometric to
$N^2\times\R$, where $N$ is a closed connected surface with $K_N\ge0$,
and
\begin{equation}\label{eq:two}
 \lim_{r\to\infty}\frac{I_p(r)}r
 =2\int_N\Scal_N\dd A=8\pi\chi(N)\le16\pi.
\end{equation}
\end{enumerate}
The upper bound in \textup{(\ref{case:one})} is attained for every
prescribed $V_M\in[0,1]$, and the constant in
\textup{(\ref{case:two})} is attained.
\end{theorem}

The theorem requires neither a pole nor a noncollapsing condition,
and assumes no additional upper or positive lower bound on $S$.
If $V_M>0$, then $M$ has one end: otherwise the product in
Theorem~\ref{thm:main}\,\textup{(\ref{case:two})} would give
$\Vol B_p(r)\le2r\Area(N)$ and hence $V_M=0$.
Nonnegativity of $S$ and \eqref{eq:main} then imply
$\chi(M)\ge V_M>0$. Since $\chi(M)\in\{0,1\}$, one has
$\chi(M)=1$, and therefore
\begin{equation}\label{eq:positive-volume-limit}
 \lim_{r\to\infty}\frac{I_p(r)}r=8\pi(1-V_M).
\end{equation}
For a one-ended manifold with $V_M=0$, the limit is
$8\pi\chi(M)\in\{0,8\pi\}$. The flat product
$\Sph^1\times\R^2$ has $\chi(M)=0$ and illustrates why the
topological term cannot always be replaced by $8\pi$.
For comparison, Liu \cite{LiuRicciThree} showed that a complete
noncompact three-manifold with $\Ric\ge0$ is either diffeomorphic
to $\R^3$ or has a universal cover that splits off a line.
Our argument treats the end topology directly through the level
surfaces of a Busemann exhaustion.

\begin{remark}
Xu \cite[Theorem~4.2]{XuSharp2026} has also obtained the exact
three-dimensional formula $8\pi(\chi(M)-V_M)$ independently. His paper proves
existence of the normalized limit in all dimensions $n\ge3$ and
sharp bounds with rigidity in dimensions $3\le n\le6$.
The present paper gives a focused three-dimensional proof using a
surface Bochner identity and separate Hessian approximations.
\end{remark}

For comparison with the asymptotic formula, recall Petrunin's
local theorem.

\begin{theorem*}[Petrunin {\cite[Theorem~1.1]{Petrunin}}]
For every $n\ge2$ there is a finite constant $C_n$, depending only on
$n$, such that every complete Riemannian $n$-manifold with $K\ge-1$
satisfies
\begin{equation}\label{eq:petrunin-local}
 \int_{B_p(1)}\Scal\dd V\le C_n
 \qquad\text{for every }p\in M.
\end{equation}
\end{theorem*}

When $K\ge0$, apply this theorem to $g_r=r^{-2}g$ and use
\[
 \int_{B_p^{g_r}(1)}\Scal_{g_r}\dd V_{g_r}
 =r^{2-n}\int_{B_p^g(r)}\Scal_g\dd V_g
\]
to obtain a bound at every radius. Thus boundedness of the normalized
integral is already known. Petrunin's argument establishes a
dimensional constant without specifying its numerical value.

Under a different geometric assumption, Ma
\cite[Theorem~1.3]{MaLCF} proved that the normalized
scalar-curvature integral has an explicit limit on every complete,
noncompact, locally conformally flat manifold with $\Ric\ge0$ in
dimensions $n\ge3$. In dimension three, the formula distinguishes
the flat and cylindrical cases from the remaining cases.

We next recall estimates under nonnegative Ricci curvature.
For complete noncompact nonparabolic three-manifolds with $\Ric\ge0$,
Xu \cite[Theorem~1.2]{XuNonparabolic} proved a weighted estimate
using the monotonicity formulas of Colding and Minicozzi
\cite{ColdingMinicozzi}. If $G_p$ is the minimal positive Green function,
normalized by $-\Delta G_p=\delta_p$, and $\beta=(4\pi G_p)^{-1}$,
with $\beta(p)=0$, then
\begin{equation}\label{eq:xu-green}
 \limsup_{t\to\infty}\frac1t
 \int_{\{\beta\le t\}}S|\nabla\beta|\dd V
 \le8\pi(1-V_M).
\end{equation}
Here nonparabolicity means that a positive Green function exists.
Both the weight and the Green-function sublevel sets are part of this
statement. Bo Zhu \cite{ZhuComparison} obtained related comparison
and integral estimates for harmonic functions. Chen, Xu, and Zhang
\cite[Theorem~1.5]{ChenXuZhang} subsequently proved that the
weighted expression in \eqref{eq:xu-green} has a limit equal to
$8\pi(1-V_M)$ when $V_M>0$.
\Needspace{5\baselineskip}
In related work, Colding \cite{ColdingMonotonicity} developed
Green-function monotonicity formulas under Ricci-curvature bounds,
and Munteanu and Wang \cite{MWGreenComparison} proved sharp
comparisons for Green-function energy and level-set area under
scalar- and Ricci-curvature hypotheses.

A pole gives direct access to geodesic balls. Recall that $o$ is a
\emph{pole} if $\exp_o:T_oM\to M$ is a diffeomorphism. Bo Zhu
\cite[Theorem~1.7]{ZhuGeometry} proved a $20\pi$ asymptotic upper
bound in dimension three under $\Ric\ge0$ and a pole. Xu later
determined the exact value.

\begin{theorem*}[Xu {\cite[Theorem~1.4]{Xu}}]
Let $(M^3,g)$ be a complete Riemannian three-manifold with
$\Ric\ge0$ and a pole. Then, for every $p\in M$,
\begin{equation}\label{eq:xu}
 \lim_{r\to\infty}\frac{I_p(r)}r=8\pi(1-V_M).
\end{equation}
\end{theorem*}

The pole supplies smooth geodesic spheres diffeomorphic to $\Sph^2$,
allowing the Gauss equation and Gauss--Bonnet theorem to be applied
without cut-locus terms. We impose the stronger condition $K\ge0$
and remove the pole assumption. Formula~\eqref{eq:main} recovers
Xu's value whenever $\chi(M)=1$, including when $V_M=0$.

Under $\Ric\ge0$ and one-endedness, Munteanu and Wang
\cite[Theorem~1.2]{MunteanuWangIntegral} proved that
$0<c\le S\le C<\infty$ implies
$\limsup_{r\to\infty}I_p(r)/r\le8\pi$.
Their proof combines smooth distance approximation, level-set
curvature identities, and geometric control from the positive
scalar-curvature bound. Our argument instead uses the convexity
available under nonnegative sectional curvature.

Deng \cite[Theorem~B(i)]{Deng2026} obtained another $8\pi$ upper
bound for $\limsup_{r\to\infty}I_p(r)/r$, assuming $\Ric\ge0$
and the additional pointwise decay
$S(x)\le C\rho(x)^{-2}$ outside a compact set. When $V_M>0$,
his proof gives the refined upper bound $8\pi(1-V_M)$
\cite[Remark~3.6]{Deng2026}. Under a different hypothesis, Deng
\cite[Theorem~B(ii)]{Deng2026} proves finiteness of the limsup
when a connected exterior region admits a proper function vanishing
on its compact boundary and satisfying $|\nabla F|=1$ throughout
its interior. Our theorem assumes neither quadratic scalar-curvature
decay nor such an exterior function and determines the limit also
when $V_M=0$. 

Pointwise lower bounds on scalar curvature also lead to complementary
volume-growth results. Munteanu and Wang \cite{MWGeometry} proved
that a complete noncompact three-manifold with $\Ric\ge0$ and
$S\ge c>0$ has at most linear volume growth. Chodosh, Li, and
Stryker \cite{CLSVolume} gave a different proof and treated
scalar-curvature lower bounds that decay at infinity. Wei, Xu, and
Zhang \cite{WXZVolume} established a sharp asymptotic linear
volume-growth bound and rigidity under $\Ric\ge0$ and a uniform
positive lower bound on $S$. These results concern volume growth
under scalar-curvature lower bounds; our theorem determines the
scalar-curvature integral under $K\ge0$ without such a lower bound.

For comparison, Chan, Lee, and Li proved the following lower bound
under nonnegative Ricci curvature and positive asymptotic volume ratio.

\begin{proposition*}[Chan--Lee--Li {\cite[Proposition~4.1]{ChanLeeLi2026}}]
Let $(M^3,g)$ be a complete, connected, noncompact Riemannian
three-manifold without boundary. If $\Ric\ge0$ and $V_M>0$, then
for every $p\in M$,
\[
 \liminf_{r\to\infty}\frac1r\int_{B_p(r)}S\dd V
 \ge8\pi(1-V_M).
\]
\end{proposition*}

\Needspace{6\baselineskip}
\begin{remark}[Scope and relation to the work of Chan--Lee--Li]
\label{rem:scope}
For a complete metric on $\R^3$ with $K\ge0$, Chan, Lee, and Li
\cite[Theorem~1.2]{ChanLeeLi2026} assume
$\alpha:=\limsup_{r\to\infty}I_p(r)/r<8\pi$ and state that
$V_M=1-\alpha/(8\pi)>0$.
Since $\chi(\R^3)=1$, Theorem~\ref{thm:main} gives
$\alpha=8\pi(1-V_M)$ for all such metrics, including $V_M=0$.
More generally, it gives an exact limit for every one-ended
three-manifold with $K\ge0$. For $\chi(M)=0$, nonnegativity of $S$
forces $V_M=0$ and the limit is zero, as in
Example~\ref{ex:flat-product}. Thus the value in the collapsed
one-ended case is $8\pi\chi(M)$.
When $V_M>0$, the preceding proposition and the upper bound in
\eqref{eq:main} also yield \eqref{eq:positive-volume-limit}.
Their lower bound holds under the weaker curvature condition
$\Ric\ge0$.

Both works use convex Busemann exhaustions and Gauss--Bonnet.
Their volume-growth argument uses Bujalo's convex-body inequality
and a radial Minkowski inequality. Our proof uses separate smooth
approximations of the Busemann and distance functions, with lower
and upper Hessian bounds, respectively. A determinant comparison
and the integrated Bochner identity give the upper estimate for
the integrated extrinsic Gaussian curvature. Its integrable negative
part and an averaged boundary estimate give the matching lower
estimate. Averaging the divergence identity then shows that the
normal Ricci integral grows sublinearly, yielding the scalar-curvature
limit directly.
The argument for the sharp upper bound was developed independently of
\cite{ChanLeeLi2026}. The limit formula is obtained by refining
that argument and retaining the Euler characteristic of the levels.
This proof yields \eqref{eq:positive-volume-limit} as a direct
consequence of the full formula, without using their lower bound.
\end{remark}

\textbf{Outline of the argument.}
Section~\ref{sec:smoothing} constructs a smooth approximation $f$
to the convex Busemann exhaustion, with a small negative Hessian
error, and a separate smooth distance approximation $u$, with a
Hessian upper bound. Their gradients approach each other, and
$|\nabla f|$ approaches one. Consequently, $f$ has no critical
points sufficiently far out. Its exterior is a product, and
one-endedness makes the large level sets connected. These surfaces
replace the geodesic spheres supplied by a pole.
Section~\ref{sec:levels} identifies their Euler characteristic and
proves the area limit. Section~\ref{sec:negative} controls the
negative curvature errors and the integrated intrinsic curvature.
Section~\ref{sec:identities} proves the normalized-gradient
divergence identity and an averaged boundary estimate. In
Section~\ref{sec:limits}, a two-dimensional determinant estimate
and the integrated Bochner formula give the upper estimate for
the integrated extrinsic curvature; the boundary average supplies
the matching lower estimate. A second use of the average proves
that the normal Ricci integral grows sublinearly. These limits give
the one-ended scalar-curvature formula directly.
Section~\ref{sec:two} treats two ends and gives the sharpness examples.
\section{Two smooth functions at infinity}\label{sec:smoothing}

Throughout this section, assume $K\ge0$. Constants denoted by $C$
may change from line to line and depend on the fixed metric,
basepoint, and chosen exhaustion, but not on the level parameter,
unless another dependence is stated explicitly.

We use the outward Busemann exhaustion based at $p$,
\begin{equation}\label{eq:busemann}
 b(x)=\sup_{\gamma}\lim_{T\to\infty}
 \bigl(T-d(x,\gamma(T))\bigr),
\end{equation}
where the supremum is over all unit-speed rays starting at $p$.
We use the following properties of $b$, recalled in
\cite[Section~3.1]{ChowLu}:
\begin{equation}\label{eq:b-asymp}
	b\le\rho,\qquad b/\rho\longrightarrow1
	\quad\text{uniformly as }\rho\to\infty,\qquad b\; \text{is 1-Lipschitz}.
\end{equation}
Here is the geometric reason for the uniform assertion. The compact
sets of initial directions of minimizing segments of length $r$
decrease, as $r\to\infty$, to the set of ray directions. Thus their
maximal angular distance $\omega(r)$ from ray directions tends to zero.
Hinge comparison, applied to a segment from $p$ to $x$ and a nearby
ray, gives $b(x)\ge\rho(x)\cos\omega(\rho(x))$ for large $\rho(x)$.
Together with $b\le\rho$, this proves \eqref{eq:b-asymp} and shows
that $b$ is a proper exhaustion bounded from below.

For the 1-Lipschitz assertion, fix a ray $\gamma$ and write $b_T(x)=T-d(x,\gamma(T))$.
The triangle inequality implies that $b_T(x)$ is nondecreasing in $T$
and satisfies $-\rho(x)\le b_T(x)\le\rho(x)$. Thus it converges to
a finite function $b_\gamma$. Each $b_T$ is $1$-Lipschitz, so
$b_\gamma$ is $1$-Lipschitz and convergence is uniform on compact sets. Since all $b_\gamma$ are $1$-Lipschitz and lie between
$-\rho$ and $\rho$, their supremum $b$ has the same bounds and is
$1$-Lipschitz.

Convexity is also a classical property of the Busemann construction on manifolds with nonnegative sectional curvature; see
\cite{CheegerGromollStructure} and \cite[Section~3.1]{ChowLu}. We give a proof here for completeness.

For clarity, Hessian comparison across a cut locus is understood
in the barrier sense, with an arbitrarily small error. More precisely,
write $r_o=d(o,\cdot)$ and fix $x\ne o$. For every $\epsilon>0$
there is a smooth upper support $r_{o,\epsilon}$ near $x$ such that
\[
 r_{o,\epsilon}\ge r_o,\qquad r_{o,\epsilon}(x)=r_o(x),\qquad
 \Hess r_{o,\epsilon}(x)
 \le\bigl(r_o(x)^{-1}+\epsilon\bigr)g_x.
\]
To obtain it, move $o$ a distance $s>0$ along a minimizing segment
from $o$ to $x$, call the new point $o_s$, and use
$s+d(o_s,\cdot)$. This function is smooth near $x$ and touches
$r_o$ from above, while its Hessian at $x$ is bounded by
$(r_o(x)-s)^{-1}g_x$. Taking $s$ sufficiently small gives the
displayed estimate. These barriers yield the weak bound
$\Hess r_o\le r_o^{-1}g$ and the corresponding upper second-derivative
bounds along geodesics. Reversing signs gives lower barriers and
the associated convexity statements. No classical Hessian of the
distance function is asserted at a cut point.

\Needspace{10\baselineskip}
\begin{lemma}[Convexity of the Busemann function]\label{lem:busemann-convex}
Suppose $K\ge0$. Each outward Busemann function
$b_\gamma(x)=\lim_{T\to\infty}(T-d(x,\gamma(T)))$ is geodesically
convex. The function $b$ in \eqref{eq:busemann}
has the same properties.
\end{lemma}

\begin{proof}
On a relatively compact open set $U$, let $R_U=\sup_U\rho$.
For $T>R_U$, fix a ray $\gamma$ and write $b_T(x)=T-d(x,\gamma(T))$. Distance Hessian comparison, with its weak extension
across the cut locus, gives
\[
 \Hess b_T\ge-\frac1{d(\cdot,\gamma(T))}g
 \ge-\frac1{T-R_U}g \quad\text{on }U.
\]
Equivalently, along any unit-speed geodesic $\sigma:[0,\ell]\to U$,
the function $b_T(\sigma(s))+s^2/(2(T-R_U))$ is convex.

Passing to the limit proves geodesic convexity of $b_\gamma$.
Taking the supremum over rays in this inequality proves convexity
of $b$.
\end{proof}

We record a form of Greene--Wu's theorem adapted to the present proof
\cite[Propositions~2.2--2.3, pp.~61--62]{GreeneWu}.
For a continuous function $\eta$ on a Riemannian manifold $U$, call a
continuous function $h$ \emph{strictly $\eta$-convex} if, for every
$x\in U$, there are a neighborhood of $x$ and a number $\delta_x>0$
on which
\[
 y\longmapsto h(y)-\tfrac12(\eta(x)+\delta_x)d(x,y)^2
\]
is geodesically convex. For smooth $h$, this condition is equivalent
to $\Hess h>\eta g$ as quadratic forms.

\begin{theorem*}[Greene--Wu fine approximation]
Let $h$ be continuous and strictly $\eta$-convex on $U$, with $\eta$
continuous. Given any continuous $\delta:U\to(0,\infty)$, there exists
$\widetilde h\in C^\infty(U)$ such that
\[
 |\widetilde h-h|<\delta,\qquad \Hess\widetilde h>\eta g.
\]
If $h$ has local Lipschitz constants strictly less than a fixed $B>0$,
the same approximation can also satisfy $|\nabla\widetilde h|<B$.
\end{theorem*}

In particular, a weak Hessian bound $\Hess h\ge\lambda g$, with
$\lambda$ continuous, implies strict $\eta$-convexity whenever
$\eta<\lambda$. This follows locally by continuity and
$\Hess(\tfrac12d(x,\cdot)^2)|_x=g_x$.
The theorem permits a positive error tolerance depending on position;
it does not require a uniform injectivity-radius lower bound.

\begin{lemma}[Smooth approximations]\label{lem:smoothing}
There exist a smooth exhaustion $f$ on $M$ and a smooth function
$u$ outside a compact set such that
\begin{align}
 |f-b|&<1,& |\nabla f|&<2,&
 \Hess f&\ge-(1+\rho)^{-3}g,\label{eq:f-smooth}\\
 |u-\rho|&<1,&&&
 \Hess u&\le(\rho^{-1}+\rho^{-2})g.\label{eq:u-smooth}
\end{align}
\end{lemma}

\begin{proof}
Apply the fine smooth approximation theorem for variable
$\eta$-convex functions, with simultaneous local Lipschitz control,
from \cite[Propositions~2.2--2.3]{GreeneWu} to $b$, which is convex
by Lemma~\ref{lem:busemann-convex}.
Use $\eta=-(1+\rho)^{-3}$ and the strict Lipschitz threshold $2$.
This gives \eqref{eq:f-smooth}; the bounded approximation error
preserves properness.

On $M\setminus\{p\}$, Hessian comparison gives
$\Hess\rho\le\rho^{-1}g$ in the weak, upper-support sense, including
the cut locus. Thus $-\rho$ satisfies the strict lower Hessian
condition with continuous threshold
$\eta=-\rho^{-1}-\rho^{-2}$. Applying the same fine approximation
theorem to $-\rho$ on an exterior open set, and reversing the sign,
gives \eqref{eq:u-smooth}. The theorem permits spatially varying
smoothing scales; no uniform injectivity radius is needed.
\end{proof}

\begin{lemma}[Asymptotic first derivatives]\label{lem:gradients}
The functions in Lemma~\ref{lem:smoothing} satisfy, uniformly at infinity,
\begin{equation}\label{eq:gradient-limits}
 \frac f\rho\longrightarrow1,\qquad
 q:=|\nabla f|\longrightarrow1,\qquad
 |\nabla u-\nabla f|\longrightarrow0.
\end{equation}
\end{lemma}

\begin{proof}
The first assertion follows from \eqref{eq:b-asymp} and
\eqref{eq:f-smooth}. Let $x$ have $\rho(x)=r$, and let
$\gamma:[0,r]\to M$ be a unit-speed minimizing geodesic from $p$ to $x$.
Integration by parts gives
\begin{align*}
 r\langle\nabla f(x),\dot\gamma(r)\rangle
 &=f(x)-f(p)+\int_0^r s\Hess f(\dot\gamma,\dot\gamma)\dd s\\
 &\ge b(x)-2-\int_0^\infty s(1+s)^{-3}\dd s.
\end{align*}
The integral equals $1/2$, so
\[
 q(x)\ge\frac{b(x)}r-\frac{5}{2r}.
\]
Together with \eqref{eq:b-asymp}, this proves that, for every
$\epsilon>0$, $q\ge1-\epsilon$ outside a compact set.

For the upper bound, the preceding estimate ensures $q(x)>0$ when
$r$ is sufficiently large. Let $\gamma_x$ be the unit-speed geodesic
determined by
\[
 \gamma_x(0)=x,\qquad
 \dot\gamma_x(0)=\frac{\nabla f(x)}{q(x)},\qquad
 \nabla_{\dot\gamma_x}\dot\gamma_x=0,
\]
and set $y=\gamma_x(r/2)$. Completeness ensures that this segment
exists; it need not minimize distance. For $0\le s\le r/2$,
\[
 \rho(\gamma_x(s))\ge r-d(x,\gamma_x(s))\ge r-s\ge r/2.
\]
Put $F(s)=f(\gamma_x(s))$. The initial velocity and the geodesic
equation give
\[
 F'(0)=q(x),\qquad
 F''(s)=\Hess f(\dot\gamma_x(s),\dot\gamma_x(s)).
\]
Since $|\dot\gamma_x|=1$, \eqref{eq:f-smooth} implies
\[
 F''(s)\ge-(1+\rho(\gamma_x(s)))^{-3}
 \ge-(1+r/2)^{-3}\ge-8r^{-3}.
\]
The first integration yields $F'(s)\ge q(x)-8s/r^3$.
Integrating once more over $[0,r/2]$ therefore gives
\begin{align*}
 f(y)-f(x)
 &=\int_0^{r/2}F'(s)\dd s\\
 &\ge\frac r2q(x)-\frac8{r^3}\int_0^{r/2}s\dd s
 =\frac r2q(x)-\frac1r.
\end{align*}
On the other hand, $f(y)\le\rho(y)+1\le3r/2+1$. Combining these
inequalities and rearranging yields
\[
 q(x)\le3-2\frac{f(x)}r+\frac2r+\frac2{r^2}.
\]
Since $f(x)/r\to1$ uniformly, this proves the upper gradient limit.

It remains to compare the two gradients. Put $w=f-u$.
The limits $f/\rho\to1$ and $|u-\rho|<1$ imply that
$|w|/\rho\to0$ uniformly. For sufficiently large $r$, the ball
$B(x,r/4)$ lies in the domain of $u$. Every $z\in B(x,r/4)$ satisfies
$3r/4\le\rho(z)\le5r/4$. Subtracting the upper Hessian bound for $u$
from the lower Hessian bound for $f$ gives
\begin{align*}
 \Hess w
 &\ge-\bigl((1+\rho)^{-3}+\rho^{-1}+\rho^{-2}\bigr)g\\
 &\ge-\frac Cr g\quad\text{on }B(x,r/4),
\end{align*}
where $C$ is independent of $x$ and $r$ for all sufficiently large $r$.
Define
\[
 m(x)=\sup_{B(x,r/4)}|w|,\qquad
 \delta(R)=\sup_{\{\rho\ge R\}}\frac{|w|}{\rho}
\]
for sufficiently large $R$. Then $\delta(R)\to0$, and the radial
bounds on this ball show that
\[
 0\le\frac{m(x)}r\le\frac54\delta(3r/4).
\]
In particular, $m(x)/r$ tends to zero uniformly as $\rho(x)=r\to\infty$.

If $\nabla w(x)=0$, the desired gradient estimate is immediate.
Otherwise, take the unit-speed geodesic $\sigma$ with initial data
\[
 \sigma(0)=x,\qquad
 \dot\sigma(0)=\frac{\nabla w(x)}{|\nabla w(x)|}.
\]
For any $0<\ell<r/4$, $d(x,\sigma(s))\le s<r/4$ on $[0,\ell]$,
so this segment remains in the ball where the Hessian bound holds.
Let $W(s)=w(\sigma(s))$. Then
\[
 W'(0)=|\nabla w(x)|,\qquad
 W''(s)=\Hess w(\dot\sigma(s),\dot\sigma(s))\ge-C/r.
\]
Integrating successively gives
\begin{align*}
 W'(s)&\ge|\nabla w(x)|-Cs/r,\\
 W(\ell)-W(0)&\ge\ell|\nabla w(x)|-\frac{C\ell^2}{2r}.
\end{align*}
Both endpoints lie in $B(x,r/4)$, so
$W(\ell)-W(0)\le|W(\ell)|+|W(0)|\le2m(x)$.
Rearranging the second inequality and dividing by $\ell$ yields
\[
 |\nabla w(x)|\le\frac{2m(x)}\ell+\frac{C\ell}{2r}.
\]
When $m(x)>0$, choose $\ell=\sqrt{m(x)r}$. The uniform convergence
of $m(x)/r$ to zero ensures $\ell<r/4$ for all sufficiently large $r$.
Substitution then gives the explicit bound
\[
 |\nabla w(x)|\le\left(2+\frac C2\right)\sqrt{\frac{m(x)}r},
\]
whose right side tends to zero uniformly.
This proves $|\nabla u-\nabla f|\to0$ uniformly at infinity.
In particular, no $C^1$ approximation of the nonsmooth distance
function has been assumed.
\end{proof}

\section{Large level sets and their area}\label{sec:levels}

Throughout Sections~\ref{sec:levels}--\ref{sec:limits},
assume that $M$ has one end.
Choose $a>2$ so large that $u$ is defined and
\[
 \tfrac12\le q\le2
 \quad\text{on }\{f\ge a\}.
\]
All levels $\Sigma_t=f^{-1}(t)$, $t\ge a$, are closed smooth surfaces.
The flow of $\nabla f/q^2$ identifies the exterior with
$\Sigma_a\times[a,\infty)$. Properness of $f$ ensures that this flow
exists across every finite slab. Every component of $\Sigma_a$
therefore gives an end: for every $T\ge a$, the components of
$\{f>T\}$ are exactly the products of the components of $\Sigma_a$
with $(T,\infty)$. Each is unbounded, and they remain distinct for
all larger $T$. Since $\{f\le T\}$ is a compact exhaustion, their
number is the number of ends. Thus each $\Sigma_t$ is connected.
In particular, this conclusion uses the absence of critical points
on the entire exterior, not merely the one-ended hypothesis.

\begin{lemma}[Euler characteristic of the end]\label{lem:end-euler}
For $t\ge a$, the number $\chi_\infty:=\chi(\Sigma_t)$ is independent
of $t$ and satisfies
\begin{equation}\label{eq:end-euler}
 \chi_\infty=2\chi(M)\in\{0,2\}.
\end{equation}
In particular, it does not depend on the choice of a smooth exhaustion
whose sufficiently large levels have the product structure above.
\end{lemma}

\begin{proof}
The level flow identifies all the $\Sigma_t$, so their Euler
characteristic is constant. Let $C_t=\{f\le t\}$. This is a compact
three-manifold with boundary $\Sigma_t$. The exterior product
structure gives a deformation retraction of $M$ onto $C_t$;
explicitly, one moves each exterior point backwards along the level
flow to level $t$ and fixes $C_t$. Thus $M$ has finite homotopy type
and $\chi(M)=\chi(C_t)$.

Double $C_t$ along its boundary. The resulting closed three-manifold
has Euler characteristic zero, by Poincar\'e duality with coefficients
in $\mathbb Z_2$, whether or not it is orientable. Additivity of the
Euler characteristic therefore gives
\[
 0=2\chi(C_t)-\chi(\Sigma_t).
\]
Hence $\chi_\infty=2\chi(M)\le2$, since $\Sigma_t$ is a closed
connected surface. It remains to show that $\chi(M)\ge0$.
By the Cheeger--Gromoll soul theorem \cite{CheegerGromollStructure},
$M$ is diffeomorphic to the normal bundle of a compact, connected,
totally geodesic soul $S_0$ of dimension at most two. Thus
$\chi(M)=\chi(S_0)$. If $\dim S_0=0$, this number is $1$;
if $\dim S_0=1$, it is $0$. If $\dim S_0=2$, the induced Gaussian
curvature is nonnegative, so Gauss--Bonnet gives $\chi(S_0)\ge0$.
Consequently $0\le\chi(M)\le1$, proving \eqref{eq:end-euler}.
\end{proof}

Let $\nu=\nabla f/q$, and let $\Pi$ denote the second fundamental
form with the convention $\Pi(X,Y)=\langle\nabla_X\nu,Y\rangle$.
Write $\kappa_1,\kappa_2$ for its principal curvatures, and set
\[
 H=\kappa_1+\kappa_2,\quad D=\kappa_1\kappa_2,\quad
 A(t)=\Area(\Sigma_t),\quad Q(t)=\int_{\Sigma_t}H\dd A.
\]
We use $K_\Sigma$ for the intrinsic Gaussian curvature on the leaf,
viewed also as a function on $\{f\ge a\}$, and $K_T$ for the ambient
sectional curvature of its tangent plane. Thus
\begin{equation}\label{eq:gauss-leaf}
 K_\Sigma=K_T+D,\qquad
 \int_{\Sigma_t}K_\Sigma\dd A=2\pi\chi_\infty\le4\pi.
\end{equation}
Gauss--Bonnet in this form also holds for nonorientable surfaces.

For tangent vectors $X,Y$, differentiating $\nu=\nabla f/q$ gives
\[
 \Pi(X,Y)=q^{-1}\Hess f(X,Y),
\]
because the term involving the derivative of $q^{-1}$ is normal.
On $\Sigma_t$, $t=f<1+\rho$ and $q\ge1/2$, so
$\Pi\ge-2(1+\rho)^{-3}g_{\Sigma_t}\ge-2t^{-3}g_{\Sigma_t}$.
Thus
\begin{equation}\label{eq:principal-lower}
 \kappa_i\ge-2t^{-3},\qquad H\ge-4t^{-3}
 \quad\text{on }\Sigma_t.
\end{equation}
The level flow has normal speed $1/q$, and therefore
\begin{equation}\label{eq:first-variation}
 A'(t)=\int_{\Sigma_t}\frac Hq\dd A\ge-8t^{-3}A(t).
\end{equation}

\begin{lemma}[Area asymptotics, including collapse]\label{lem:area}
One has
\begin{equation}\label{eq:area}
 \lim_{t\to\infty}\frac{A(t)}{t^2}=4\pi{V_M}.
\end{equation}
In particular, there is a constant $C$ such that $A(t)\le Ct^2$
for all $t\ge a$, including when ${V_M}=0$.
\end{lemma}

\begin{proof}
Write $V_f(t)=\Vol\{f<t\}$. Given $0<\delta<1$, outside a compact
set one has $(1-\delta)\rho\le f\le(1+\delta)\rho$.
For sufficiently large $t$, boundedness of $f$ on that compact set
therefore gives
\[
 B_p\bigl(t/(1+\delta)\bigr)\subset\{f<t\}
 \subset B_p\bigl(t/(1-\delta)\bigr).
\]
Divide volumes by $t^3$, pass to lower and upper limits, and let
$\delta\downarrow0$. This proves
\[
 \lim_{t\to\infty}\frac{V_f(t)}{t^3}=\frac{4\pi}{3}{V_M}.
\]
Set $\Om a t=\{a<f<t\}$. The coarea formula gives
\begin{equation}\label{eq:integrated-area}
 F_0(t):=\int_a^t A(s)\dd s=\int_{\Om a t}q\dd V,
 \qquad
 \lim_{t\to\infty}\frac{F_0(t)}{t^3}=\frac{4\pi}{3}{V_M}.
\end{equation}
Indeed, for a fixed large $L$, the difference between the volume
integral and $\Vol(\Om a t)$ has absolute value at most
\[
 C_L+\sup_{\{f\ge L\}}|q-1|\,V_f(t).
\]
Divide by $t^3$, let $t\to\infty$, and then let $L\to\infty$.

By \eqref{eq:first-variation}, $A(t)\exp(\int_a^t8s^{-3}\dd s)$
is nondecreasing. Fix $0<\epsilon<1$ and define
\[
 E_+(t)=\exp\left(\int_t^{(1+\epsilon)t}8s^{-3}\dd s\right),
 \qquad
 E_-(t)=\exp\left(\int_{(1-\epsilon)t}^t8s^{-3}\dd s\right).
\]
Both tend to $1$. For sufficiently large $t$, monotonicity gives
\begin{align*}
 A(t)&\le E_+(t)\frac{F_0((1+\epsilon)t)-F_0(t)}{\epsilon t},\\
 A(t)&\ge E_-(t)^{-1}\frac{F_0(t)-F_0((1-\epsilon)t)}{\epsilon t}.
\end{align*}
Consequently,
\begin{align*}
 \limsup_{t\to\infty}\frac{A(t)}{t^2}
 &\le\frac{4\pi{V_M}}{3}\frac{(1+\epsilon)^3-1}{\epsilon},\\
 \liminf_{t\to\infty}\frac{A(t)}{t^2}
 &\ge\frac{4\pi{V_M}}{3}\frac{1-(1-\epsilon)^3}{\epsilon}.
\end{align*}
Letting $\epsilon\downarrow0$ proves the assertion. If ${V_M}=0$,
the upper bound and $A(t)\ge0$ suffice.
\end{proof}

\section{Integrability of the negative curvature errors}\label{sec:negative}

For a real-valued function $h$, write $h_+=\max\{h,0\}$ and
$h_-=\max\{-h,0\}$.

\begin{lemma}\label{lem:negative}
On the exterior foliated by the $\Sigma_t$,
\begin{equation}\label{eq:negative-integrable}
 \int_{\{f>a\}}\bigl(D_-+(K_\Sigma)_-\bigr)\dd V<\infty.
\end{equation}
The corresponding integrals with respect to $\dd t\dd A$ are also finite.
Moreover, for some constant $C$ and all $t\ge a$,
$\int_{\Sigma_t}H_-\dd A\le C/t$.
\end{lemma}

\begin{proof}
Put $z=f^{-3}$. On the exterior, $\Hess f\ge-zg$, so
$\Delta f+3z\ge0$. Taking the trace in an adapted orthonormal frame $e_1,e_2,\nu$ gives
\[
 \Delta f=qH+\Hess f(\nu,\nu).
\]
Thus $qH\le\Delta f+z$. Since $\Delta f+3z\ge0$,
$qH_+\le(\Delta f+z)_+\le\Delta f+3z$, and hence
\[
 H_+\le2(\Delta f+3z).
\]
By \eqref{eq:principal-lower}, both $\kappa_i+2z$ are nonnegative.
Expanding their product gives
\[
 D_-\le2zH_++4z^2\le4z(\Delta f+3z)+4z^2.
\]
Since $K_T\ge0$, \eqref{eq:gauss-leaf} implies
$(K_\Sigma)_-\le D_-$.

Let $F(t)=\int_{\Sigma_t}q\dd A$. Lemma~\ref{lem:area} implies
$F(t)\le Ct^2$ for some constant $C$. Integration by parts gives
\begin{equation}\label{eq:weighted-laplacian}
 \int_{\Om a T}f^{-3}\Delta f\dd V
 =T^{-3}F(T)-a^{-3}F(a)
   +3\int_{\Om a T}f^{-4}q^2\dd V.
\end{equation}
By the coarea formula and $F(t)\le Ct^2$,
\[
 \int_{\{f>a\}}f^{-4}q^2\dd V
 =\int_a^\infty t^{-4}F(t)\dd t
 \le C\int_a^\infty t^{-2}\dd t<\infty.
\]
Also, $q\ge1/2$ and $A(t)\le Ct^2$ imply
\[
 \int_{\{f>a\}}f^{-6}\dd V
 \le 2C\int_a^\infty t^{-4}\dd t<\infty.
\]
The boundary term satisfies $T^{-3}F(T)\le C/T$. Thus the integrals
of the nonnegative function $z(\Delta f+3z)$ over $\Om a T$ are
uniformly bounded. Monotone convergence and the estimate for $D_-$
prove \eqref{eq:negative-integrable}. The coarea formula gives
$\dd t\dd A=q\dd V$, so the corresponding negative parts in this
measure are integrable as well. Finally,
\[
 \int_{\Sigma_t}H_-\dd A\le4t^{-3}A(t)\le C/t.
\]
\end{proof}

\paragraph{Changing the coarea measure.}
We will use the following elementary observation twice. Write
$\dd\mu=\dd t\dd A=q\dd V$. If a function $h$ on the exterior
satisfies $\int_{\Om a T}|h|\dd\mu\le CT$ for all large $T$, then
\begin{equation}\label{eq:measure-change}
 \lim_{T\to\infty}\frac1T
 \left|\int_{\Om a T}h\dd V-\int_{\Om a T}h\dd\mu\right|=0.
\end{equation}
In fact, for any fixed $L>a$ and $T\ge L$, the expression before
division by $T$ is bounded by
$C_L+CT\sup_{\{f\ge L\}}|q^{-1}-1|$.
First let $T\to\infty$, and then $L\to\infty$.

\begin{lemma}[Integrated intrinsic curvature]\label{lem:intrinsic}
Set $G(T)=\int_{\Om a T}K_\Sigma\dd V$. Then
\begin{equation}\label{eq:intrinsic}
 \lim_{T\to\infty}\frac{G(T)}T=2\pi\chi_\infty,
 \qquad \int_{\Om a T}|K_\Sigma|\dd V\le CT
\end{equation}
for some constant $C$ and all sufficiently large $T$.
\end{lemma}

\begin{proof}
Since the Euler characteristic of the levels is constant,
Gauss--Bonnet gives the exact identity
\[
 \int_{\Om a T}K_\Sigma\dd\mu=2\pi\chi_\infty(T-a).
\]
By Lemma~\ref{lem:negative}, there is a constant $C_0$ such that
the integral of $(K_\Sigma)_-$ with respect to $\dd\mu$ is at most
$C_0$ on every slab. By Lemma~\ref{lem:end-euler},
$\chi_\infty\in\{0,2\}$.

The identity $|h|=h+2h_-$ now gives
\[
 \int_{\Om a T}|K_\Sigma|\dd\mu
 \le4\pi(T-a)+2C_0\le CT
\]
for all sufficiently large $T$. Since $q^{-1}\le2$, the integral
with respect to $\dd V$ is at most $2CT$. Apply
\eqref{eq:measure-change} to the signed integral and divide the
Gauss--Bonnet identity by $T$ to obtain the limit in
\eqref{eq:intrinsic}.
\end{proof}

\section{A divergence identity and boundary averages}\label{sec:identities}

The following lemma is used in \cite{Petrunin}. We give a proof here for completeness.

\begin{lemma}[A divergence identity]\label{lem:divergence}
Set $P(T)=\int_{\Om a T}D\dd V$ and
$J(T)=\int_{\Om a T}\Ric(\nu,\nu)\dd V$. Then
\begin{equation}\label{eq:ricci-identity}
 J(T)=2P(T)+Q(a)-Q(T),
\end{equation}
where $Q(t)=\int_{\Sigma_t}H\dd A$.
Consequently, with $I_f(T)=\int_{\{f<T\}}S\dd V$,
\begin{equation}\label{eq:scalar-identity}
 I_f(T)=C_a+2G(T)+2P(T)-2Q(T),
 \qquad C_a=\int_{\{f<a\}}S\dd V+2Q(a).
\end{equation}
Equivalently,
\begin{equation}\label{eq:scalar-reduced}
 I_f(T)=I_f(a)+2G(T)-2P(T)+2J(T).
\end{equation}
\end{lemma}

\begin{proof}
Put $Z=\nabla_\nu\nu$. Since $|\nu|=1$, every $\nabla_X\nu$ is
orthogonal to $\nu$; in particular, $Z$ is tangent to the leaves.
In an adapted orthonormal frame $e_1,e_2,\nu$,
\[
 \diver\nu=\sum_{i=1}^2\langle\nabla_{e_i}\nu,e_i\rangle=H.
\]
At a point, compute in a normal orthonormal frame, using repeated
indices for summation. Differentiating $Z^i=\nu^j\nabla_j\nu^i$ gives
\[
 \diver Z=(\nabla_i\nu^j)(\nabla_j\nu^i)
             +\nu^j\nabla_i\nabla_j\nu^i.
\]
Commuting the two covariant derivatives and contracting the output
index with the first derivative index gives
\[
 \nabla_i\nabla_j\nu^i
 =\nabla_j\nabla_i\nu^i+\Ric_{kj}\nu^k.
\]
To evaluate the quadratic term, consider the endomorphism
$L(X)=\nabla_X\nu$. In the adapted orthonormal frame it has block
matrix
\[
 L=\begin{pmatrix}\Pi^\sharp&Z\\0&0\end{pmatrix},
\]
where $\Pi^\sharp$ denotes the shape operator. The last row is zero
because $\langle\nabla_X\nu,\nu\rangle=0$. Hence
\[
 (\nabla_i\nu^j)(\nabla_j\nu^i)
 =\tr(L^2)=\tr\bigl((\Pi^\sharp)^2\bigr)=|\Pi|^2.
\]
In particular, the contraction is the trace of the square of $L$,
not its squared norm, so there is no $|Z|^2$ term. It follows that
\[
 \diver Z=|\Pi|^2+\nu(H)+\Ric(\nu,\nu).
\]

On the other hand, $\diver(H\nu)=\nu(H)+H^2$. Subtraction gives
\[
 \diver(H\nu-Z)=H^2-|\Pi|^2-\Ric(\nu,\nu)
 =2D-\Ric(\nu,\nu).
\]

The outward unit normals of $\Om a T$ are $\nu$ on $\Sigma_T$ and
$-\nu$ on $\Sigma_a$. Since $Z$ is tangential, the corresponding
fluxes of $H\nu-Z$ are $H$ and $-H$. The divergence theorem gives
$Q(T)-Q(a)=2P(T)-J(T)$, which proves \eqref{eq:ricci-identity}.

The three-dimensional Gauss equation reads
\[
 S=2K_\Sigma+2\Ric(\nu,\nu)-2D.
\]
Integrating and substituting \eqref{eq:ricci-identity} gives
\eqref{eq:scalar-identity}. Alternatively, integrating this Gauss
equation directly gives \eqref{eq:scalar-reduced}.
\end{proof}

We also record an estimate that uses only the area variation and
the negative mean-curvature bound. For sufficiently large $R$, set
\[
 \underline q_R=\min\bigl\{1,\inf_{\{f\ge R\}}q\bigr\}.
\]
Then $\underline q_R\to1$ and $0\le1-\underline q_R/q\le1$ on $\{f\ge R\}$.
By \eqref{eq:first-variation} and Lemma~\ref{lem:negative},
for $t\ge R$ we have
\[
 Q(t)-\underline q_RA'(t)
 =\int_{\Sigma_t}H(1-\underline q_R/q)\dd A
 \ge-\int_{\Sigma_t}H_-\dd A\ge-C/t.
\]
Integration over $[R,T]$ therefore gives
\begin{equation}\label{eq:boundary-average}
 \int_R^T Q(t)\dd t
 \ge \underline q_R\bigl(A(T)-A(R)\bigr)-C\log(T/R)
 \qquad(T>R).
\end{equation}
In particular, for every fixed $\lambda>1$, Lemma~\ref{lem:area}
implies
\begin{equation}\label{eq:boundary-average-limit}
 \liminf_{R\to\infty}\frac1{R^2}\int_R^{\lambda R}Q(t)\dd t
 \ge4\pi V_M(\lambda^2-1).
\end{equation}
These estimates will give both the lower bound for $P(T)/T$ and
the vanishing limit of $J(T)/T$.

\section{Curvature limits and the one-ended formula}\label{sec:limits}

\begin{proposition}[Limit of the integrated extrinsic curvature]\label{prop:extrinsic}
Set $P(T)=\int_{\Om a T}D\dd V$. Then
\begin{equation}\label{eq:extrinsic}
 \lim_{T\to\infty}\frac{P(T)}T=4\pi V_M,
 \qquad \int_{\Om a T}|D|\dd V\le CT
\end{equation}
for some constant $C$ and all sufficiently large $T$.
\end{proposition}

\begin{proof}
We first prove the upper estimate and the absolute integral bound.
On $\Sigma_t$ set $v=u|_{\Sigma_t}$ and
$\alpha=\langle\nabla u,\nu\rangle$. Lemma~\ref{lem:gradients} gives
\begin{equation}\label{eq:alpha}
 \alpha\longrightarrow1,\qquad
 \eta_t:=\sup_{\Sigma_t}|\nabla_{\Sigma_t}v|\longrightarrow0,
\end{equation}
uniformly. Enlarge $a$ if necessary so that $\alpha>0$ there.
The Gauss formula, with our sign convention, is
$\nabla_X^M Y=\nabla_X^{\Sigma_t}Y-\Pi(X,Y)\nu$.
Substituting it in the definition of the Hessian gives
\begin{equation}\label{eq:restriction}
 \Hess_{\Sigma_t}v=(\Hess u)|_{T\Sigma_t}-\alpha\Pi.
\end{equation}
Let $e_t=2t^{-3}$. Then $\Pi+e_tg_{\Sigma_t}$ is positive
semidefinite. By \eqref{eq:u-smooth}, \eqref{eq:alpha}, and
$\rho/t\to1$ on $\Sigma_t$, there is a scalar
\[
 \lambda_t>0,\qquad \lim_{t\to\infty}t\lambda_t=1
\]
constant on each leaf, such that the following inequalities hold as
quadratic forms:
\begin{equation}\label{eq:matrix}
 0\le\alpha(\Pi+e_tg_{\Sigma_t})
 \le\lambda_tg_{\Sigma_t}-\Hess_{\Sigma_t}v.
\end{equation}
For example, take the supremum on the leaf of
$\rho^{-1}+\rho^{-2}+\alpha e_t$ for $\lambda_t$.

We compute the following determinants in normal coordinates for
the induced metric on $\Sigma_t$ at each point. Monotonicity of
the determinant on positive semidefinite $2\times2$ matrices,
together with $H\ge-2e_t$, implies
\begin{align*}
 D&=\det(\Pi+e_tI)-e_tH-e_t^2\\
  &\le\alpha^{-2}\det(\lambda_tI-\Hess_{\Sigma_t}v)+e_t^2.
\end{align*}
The determinant on the right is nonnegative by \eqref{eq:matrix}.
Put $c_t=(\min_{\Sigma_t}\alpha)^{-2}$, so $c_t\to1$. Integrating,
expanding the two-dimensional determinant, and using
$\int_{\Sigma_t}\Delta_{\Sigma_t}v\dd A=0$, we obtain
\begin{equation}\label{eq:det-bound}
 \int_{\Sigma_t}D\dd A
 \le c_t\left(\lambda_t^2A(t)
       +\int_{\Sigma_t}\det\Hess_{\Sigma_t}v\dd A\right)
       +e_t^2A(t).
\end{equation}
The integrated Bochner identity on a closed surface gives
\begin{align}
 \int_{\Sigma_t}\det\Hess_{\Sigma_t}v\dd A
 &=\frac12\int_{\Sigma_t}
       \bigl((\Delta_{\Sigma_t}v)^2-|\Hess_{\Sigma_t}v|^2\bigr)\dd A
       \notag\\
 &=\frac12\int_{\Sigma_t}K_\Sigma|\nabla_{\Sigma_t}v|^2\dd A
       \notag\\
 &\le\frac{\eta_t^2}{2}
       \left(4\pi+N(t)\right),\qquad
 N(t)=\int_{\Sigma_t}(K_\Sigma)_-\dd A.\label{eq:surface-bochner}
\end{align}
The last inequality uses
$\int(K_\Sigma)_+\le4\pi+\int(K_\Sigma)_-$.

By Lemma~\ref{lem:area},
\[
 \lim_{t\to\infty}\lambda_t^2A(t)=4\pi{V_M},
 \qquad e_t^2A(t)\le Ct^{-4}.
\]
The function
\[
 B(t)=c_t\bigl(\lambda_t^2A(t)+2\pi\eta_t^2\bigr)+e_t^2A(t)
\]
therefore tends to $4\pi{V_M}$. Also $c_t\eta_t^2$ is bounded
and $\int_a^\infty N(t)\dd t<\infty$. Combining
\eqref{eq:det-bound} and \eqref{eq:surface-bochner}, then integrating,
gives
\[
 \int_{\Om a T}D\dd\mu
 \le\int_a^T B(t)\dd t+C\int_a^T N(t)\dd t.
\]
Dividing by $T$ and using the limit of the averages of $B$ yields
\[
 \limsup_{T\to\infty}T^{-1}\int_{\Om a T}D\dd\mu
 \le4\pi{V_M}.
\]
The negative part of $D$ has bounded integral by
Lemma~\ref{lem:negative}, so $\int_{\Om a T}|D|\dd\mu\le CT$
for all sufficiently large $T$. Equation~\eqref{eq:measure-change}
transfers the signed estimate to $\dd V$, and $q^{-1}\le2$ transfers
the absolute bound. In particular,
\begin{equation}\label{eq:extrinsic-upper}
 \limsup_{T\to\infty}\frac{P(T)}T\le4\pi V_M.
\end{equation}

To prove the reverse estimate, use Lemma~\ref{lem:negative} to set
\[
 C_-:=\int_{\{f>a\}}D_-\dd V<\infty.
\]
For $T\ge s\ge a$, one has
\begin{equation}\label{eq:p-almost-monotone}
 P(T)-P(s)=\int_{\{s<f<T\}}D\dd V\ge-C_-.
\end{equation}
Moreover, $\Ric\ge0$ and \eqref{eq:ricci-identity} give
\begin{equation}\label{eq:p-q-lower}
 2P(s)=J(s)+Q(s)-Q(a)\ge Q(s)-Q(a).
\end{equation}
Consequently, for every $s\in[R,T]$,
\[
 2P(T)\ge Q(s)-Q(a)-2C_-.
\]
Integrating this inequality over $[R,T]$ and applying
\eqref{eq:boundary-average} yields
\begin{equation}\label{eq:extrinsic-lower-average}
 \begin{split}
 2(T-R)P(T)
 &\ge \underline q_R\bigl(A(T)-A(R)\bigr)-C\log(T/R)\\
 &\quad-(T-R)\bigl(Q(a)+2C_-\bigr).
 \end{split}
\end{equation}
Fix $\lambda>1$ and put $T=\lambda R$. Divide by
$2(T-R)T$. Since $\underline q_R\to1$ and
$A(t)/t^2\to4\pi V_M$, the area term tends to
\[
 \frac{4\pi V_M(\lambda^2-1)}{2\lambda(\lambda-1)}
 =\frac{2\pi V_M(\lambda+1)}{\lambda},
\]
whereas the other terms tend to zero. Hence
\[
 \liminf_{T\to\infty}\frac{P(T)}T
 \ge\frac{2\pi V_M(\lambda+1)}{\lambda}.
\]
Letting $\lambda\downarrow1$ and combining this with
\eqref{eq:extrinsic-upper} proves the limit in
\eqref{eq:extrinsic}. The argument also applies when $V_M=0$.
\end{proof}

\begin{lemma}[Limits of the normal Ricci integral and boundary term]
\label{lem:boundary-limits}
One has
\begin{equation}\label{eq:boundary-limits}
 \lim_{T\to\infty}\frac{J(T)}T=0,
 \qquad
 \lim_{T\to\infty}\frac{Q(T)}T=8\pi V_M.
\end{equation}
\end{lemma}

\begin{proof}
Since $\Ric\ge0$, the function $J$ is nonnegative and
nondecreasing. Thus \eqref{eq:ricci-identity} gives
\[
 T J(T)\le\int_T^{2T}J(t)\dd t
 =2\int_T^{2T}P(t)\dd t+TQ(a)-\int_T^{2T}Q(t)\dd t.
\]
Apply \eqref{eq:boundary-average} with endpoints $T$ and $2T$
and divide by $T^2$. We obtain
\begin{equation}\label{eq:normal-ricci-average}
 \begin{split}
 0\le\frac{J(T)}T
 &\le\frac2{T^2}\int_T^{2T}P(t)\dd t+\frac{Q(a)}T\\
 &\quad-\underline q_T\frac{A(2T)-A(T)}{T^2}
       +\frac{C\log2}{T^2}.
 \end{split}
\end{equation}
By Proposition~\ref{prop:extrinsic},
\[
 \lim_{T\to\infty}\frac2{T^2}\int_T^{2T}P(t)\dd t
 =12\pi V_M.
\]
Indeed, for every $\delta>0$ and all sufficiently large $T$,
$|P(t)-4\pi V_M t|\le\delta t$ for $t\in[T,2T]$; the resulting
error after integration and multiplication by $2/T^2$ is at most
$3\delta$. Since $\underline q_T\to1$, Lemma~\ref{lem:area} also gives
\[
 \lim_{T\to\infty}\underline q_T\frac{A(2T)-A(T)}{T^2}
 =12\pi V_M.
\]
The right side of \eqref{eq:normal-ricci-average} therefore tends
to zero, proving $J(T)/T\to0$. Finally,
\eqref{eq:ricci-identity} yields
\[
 \frac{Q(T)}T=2\frac{P(T)}T+\frac{Q(a)}T-\frac{J(T)}T.
\]
The limit of $P(T)/T$ and the first conclusion now give
$Q(T)/T\to8\pi V_M$.
\end{proof}

\begin{proof}[Proof of Theorem~\ref{thm:main}\,\textup{(\ref{case:one})}]
By Lemma~\ref{lem:end-euler},
$\chi_\infty=2\chi(M)$ and $\chi(M)\in\{0,1\}$. Divide
\eqref{eq:scalar-identity} by $T$ and use
Lemma~\ref{lem:intrinsic}, Proposition~\ref{prop:extrinsic},
and Lemma~\ref{lem:boundary-limits}. The result is
\begin{equation}\label{eq:sublevel-limit}
 \lim_{T\to\infty}\frac{I_f(T)}T
 =4\pi\chi_\infty-8\pi V_M=:L.
\end{equation}
In particular, $L\ge0$ because $S\ge0$.

It remains to pass from sublevel sets to geodesic balls. Fix
$0<\delta<1$. Lemma~\ref{lem:gradients} gives
$(1-\delta)\rho\le f\le(1+\delta)\rho$ outside a compact set $K$.
Choose $r$ large enough that $K\subset B_p(r)$ and
$\max_K f<(1-\delta)r$. Then
\[
 \{f<(1-\delta)r\}\subset B_p(r)
 \subset\{f<(1+\delta)r\}.
\]
Indeed, the first inclusion follows from the lower bound on $f$
outside $K$, and the second from the upper bound there and the
choice of $r$ on $K$. Nonnegativity of $S$ therefore gives
\[
 I_f((1-\delta)r)\le I_p(r)\le I_f((1+\delta)r).
\]
After division by $r$, \eqref{eq:sublevel-limit} implies
\[
 (1-\delta)L
 \le\liminf_{r\to\infty}\frac{I_p(r)}r
 \le\limsup_{r\to\infty}\frac{I_p(r)}r
 \le(1+\delta)L.
\]
Letting $\delta\downarrow0$ gives the limit $L$ over geodesic balls.
Since $\chi_\infty=2\chi(M)$, we have
$L=8\pi(\chi(M)-V_M)$, proving \eqref{eq:main}.
Its upper bound follows from $\chi(M)\le1$.
\end{proof}

\section{Two ends and sharp examples}\label{sec:two}

\begin{proof}[Proof of Theorem~\ref{thm:main}\,\textup{(\ref{case:two})}]
By the Cheeger--Gromoll
splitting theorem \cite{CheegerGromoll}, a complete manifold with nonnegative Ricci curvature has at most two
ends. If there are two, we have $M=N^2\times\R$ with $N$
closed and connected. The induced metric on $N$ has $K_N\ge0$.
Let $p=(x_0,0)$ and $d_N=\operatorname{diam}(N)$. For $R>d_N$,
\[
 N\times\bigl(-\sqrt{R^2-d_N^2},\sqrt{R^2-d_N^2}\bigr)
 \subset B_p(R)\subset N\times(-R,R).
\]
The scalar curvature on the product is $S_N=2K_N\ge0$. Integrating
over these inclusions and dividing by $R$ yields
\[
 \lim_{R\to\infty}\frac{I_p(R)}R=2\int_N S_N\dd A
 =8\pi\chi(N)\le16\pi.
\]
The sharpness assertions are verified below.
\end{proof}

\paragraph{Equality for every asymptotic volume ratio.}
Fix $\ell\in[0,1]$. Choose a smooth nonincreasing function
$h:[0,\infty)\to[\ell,1]$ equal to $1$ near zero and to $\ell$
for $s\ge s_0$. Define
\[
 F(s)=\int_0^s h(\tau)\dd\tau,\qquad
 g=\dd s^2+F(s)^2g_{\Sph^2}
\]
on $\R^3$, collapsing the sphere at $s=0$. Since $F(s)=s$ near zero,
the metric is smooth at the origin $o$. It is complete and one-ended,
and $s$ is the distance from $o$. Its sectional curvatures are
\[
 K_{\mathrm{rad}}=-F''/F\ge0,
 \qquad K_{\mathrm{tan}}=(1-(F')^2)/F^2\ge0.
\]
For $s\ge s_0$, $F(s)=\ell s+c$ with a constant $c\ge0$.
Thus
\[
 {V_M}=\lim_{R\to\infty}
 \frac{4\pi\int_0^R F(s)^2\dd s}{(4\pi/3)R^3}=\ell^2.
\]
The scalar-curvature density in the radial variable is
\[
 S(s)\,4\pi F(s)^2
 =8\pi\bigl(1-(F'(s))^2-2F(s)F''(s)\bigr).
\]
For $s\ge s_0$ this is exactly $8\pi(1-\ell^2)$. Therefore
\[
 I_o(R)=I_o(s_0)+8\pi(1-\ell^2)(R-s_0)\quad(R\ge s_0),
 \qquad
 \lim_{R\to\infty}\frac{I_o(R)}R=8\pi(1-{V_M}).
\]
Since $M$ is diffeomorphic to $\R^3$, one has $\chi(M)=1$.
Thus the calculation agrees with the full
limit formula in Theorem~\ref{thm:main}. As $\ell$ ranges over $[0,1]$,
it realizes equality in the one-ended upper bound for every
$V_M\in[0,1]$. The case $\ell=0$ has a cylindrical end;
the case $\ell=1$ is Euclidean space. The two-ended constant
$16\pi$ is attained on $\Sph_b^2\times\R$ for every $b>0$.

\begin{example}[A collapsed one-ended flat product]\label{ex:flat-product}
Let $\Sph^1_L$ be a circle of length $L>0$, and equip
$M=\Sph^1_L\times\R^2$ with the product metric. This manifold is
complete, connected, noncompact, without boundary, and flat.
For every $R>0$, the complement of the compact set
\[
 C_R=\Sph^1_L\times\overline{B_{\R^2}(0,R)}
\]
is connected and unbounded. Since the $C_R$ form a compact
exhaustion, $M$ has one end. Moreover,
$\pi_1(M)\cong\mathbb Z$, so $M$ is not homeomorphic to $\R^3$.

For $p=(z_0,x_0)$ and every $r>0$, projection to the Euclidean
factor gives
\[
 B_p(r)\subset\Sph^1_L\times B_{\R^2}(x_0,r),
 \qquad \Vol B_p(r)\le L\pi r^2.
\]
Consequently,
\[
 0\le\frac{\Vol B_p(r)}{(4\pi/3)r^3}
 \le\frac{3L}{4r}\longrightarrow0,
\]
and $V_M=0$. Since $S\equiv0$, we have $I_p(r)=0$ for every $r>0$,
and hence
\[
 \lim_{r\to\infty}\frac{I_p(r)}r
 =0<8\pi=8\pi(1-V_M).
\]
The boundary of $C_R$ is a torus, and $M$ deformation retracts
onto a circle. Hence $\chi(M)=0$, in agreement with the limit
formula \eqref{eq:main}. The example shows that the positive-volume
formula \eqref{eq:positive-volume-limit} does not extend to all
collapsed one-ended manifolds. The term $8\pi\chi(M)$ in
\eqref{eq:main} records this topological distinction.
\end{example}

Finally, the asymptotic limit is independent of the basepoint.
If $d(p,p')=d$, then for $R>d$, nonnegativity of $S$ and ball inclusions
give
\[
 I_p(R-d)\le I_{p'}(R)\le I_p(R+d).
\]
If $I_p(R)/R$ tends to $L$, then both
$I_p(R\pm d)/R$ tend to $L$, since $(R\pm d)/R\to1$.
The displayed inclusions therefore give the same limit at $p'$.


\begin{thebibliography}{99}

\bibitem{ChanLeeLi2026}
P.-Y.~Chan, M.-C.~Lee, and M.~Li,
\emph{Volume growth and integral curvature bound for
	non-negatively curved three-manifolds},
preprint (2026), arXiv:2609.24096v1.
\url{https://arxiv.org/abs/2609.24096v1}.

\bibitem{CheegerGromoll}
J.~Cheeger and D.~Gromoll,
\emph{The splitting theorem for manifolds of nonnegative Ricci curvature},
J. Differential Geom. \textbf{6} (1971), 119--128.

\bibitem{CheegerGromollStructure}
J.~Cheeger and D.~Gromoll,
\emph{On the structure of complete manifolds of nonnegative curvature},
Ann. of Math. (2) \textbf{96} (1972), no.~3, 413--443.

\bibitem{ChenXuZhang}
Z.~Chen, G.~Xu, and S.~Zhang,
\emph{Integrals and rigidity on manifolds with nonnegative Ricci curvature},
preprint (2026), arXiv:2602.10393.
\url{https://arxiv.org/abs/2602.10393}.

\bibitem{Cheng}
H.~Cheng,
\emph{Unbounded normalized scalar curvature integrals in dimension four},
preprint (2026), arXiv:2609.10323.
\url{https://arxiv.org/abs/2609.10323}.

\bibitem{CLSVolume}
O.~Chodosh, C.~Li, and D.~Stryker,
\emph{Volume growth of $3$-manifolds with scalar curvature lower bounds},
Proc. Amer. Math. Soc. \textbf{151} (2023), no.~10, 4501--4511.
\url{https://doi.org/10.1090/proc/16521}.

\bibitem{ChowLu}
B.~Chow and P.~Lu,
\emph{On the asymptotic scalar curvature ratio of complete Type I-like
ancient solutions to the Ricci flow on noncompact $3$-manifolds},
Comm. Anal. Geom. \textbf{12} (2004), no.~1--2, 59--91.
\url{https://intlpress.com/JDetail/1805786063895486465}.

\bibitem{CohnVossen}
S.~Cohn-Vossen,
\emph{K\"urzeste Wege und Totalkr\"ummung auf Fl\"achen},
Compositio Math. \textbf{2} (1935), 69--133.
\url{https://www.numdam.org/item/CM_1935__2__69_0/}.

\bibitem{ColdingMonotonicity}
T.~H.~Colding,
\emph{New monotonicity formulas for Ricci curvature and applications: I},
Acta Math. \textbf{209} (2012), no.~2, 229--263.
\url{https://doi.org/10.1007/s11511-012-0086-2}.

\bibitem{ColdingMinicozzi}
T.~H.~Colding and W.~P.~Minicozzi~II,
\emph{Ricci curvature and monotonicity for harmonic functions},
Calc. Var. Partial Differential Equations \textbf{49} (2014),
no.~3--4, 1045--1059.
\url{https://doi.org/10.1007/s00526-013-0610-z}.

\bibitem{Deng2026}
J.~Deng,
\emph{Cohn--Vossen-type inequalities for three-manifolds and
locally conformally flat manifolds},
preprint (2026), arXiv:2606.01368v2.
\url{https://arxiv.org/abs/2606.01368v2}.

\bibitem{GreeneWu}
R.~E.~Greene and H.~Wu,
\emph{$C^\infty$ approximations of convex, subharmonic, and
plurisubharmonic functions},
Ann. Sci. \'Ecole Norm. Sup. (4) \textbf{12} (1979), 47--84,
especially Propositions~2.2--2.3.
\url{https://www.numdam.org/item/ASENS_1979_4_12_1_47_0/}.

\bibitem{HaoZhu}
T.~Hao and J.~Zhu,
\emph{A counterexample to Yau's conjectured asymptotic scalar-curvature
integral bound},
preprint (2026), arXiv:2609.06533.
\url{https://arxiv.org/abs/2609.06533}.

\bibitem{LiuRicciThree}
G.~Liu,
\emph{$3$-manifolds with nonnegative Ricci curvature},
Invent. Math. \textbf{193} (2013), no.~2, 367--375.
\url{https://doi.org/10.1007/s00222-012-0428-x}.

\bibitem{MaLCF}
S.~Ma,
\emph{Cohn--Vossen theory for locally conformally flat manifolds},
preprint (2026), arXiv:2606.13979v2.
\url{https://arxiv.org/abs/2606.13979v2}.

\bibitem{MWGreenComparison}
O.~Munteanu and J.~Wang,
\emph{Comparison theorems for $3$D manifolds with scalar curvature bound},
Int. Math. Res. Not. IMRN (2023), no.~3, 2215--2242.
\url{https://doi.org/10.1093/imrn/rnab307}.

\bibitem{MWGeometry}
O.~Munteanu and J.~Wang,
\emph{Geometry of three-dimensional manifolds with positive scalar curvature},
Amer. J. Math. \textbf{148} (2026), no.~1, 131--160.
\url{https://doi.org/10.1353/ajm.2026.a980770}.

\bibitem{MunteanuWangIntegral}
O.~Munteanu and J.~Wang,
\emph{Sharp integral bound of scalar curvature on $3$-manifolds},
Trans. Amer. Math. Soc., published electronically (2026).
\url{https://doi.org/10.1090/tran/9744}.

\bibitem{Petrunin}
A.~Petrunin,
\emph{An upper bound for the curvature integral},
St. Petersburg Math. J. \textbf{20} (2009), 255--265;
Russian original in Algebra i Analiz \textbf{20} (2008), no.~2, 134--148.
\url{https://doi.org/10.1090/S1061-0022-09-01046-2}.

\bibitem{ShiYau}
W.-X.~Shi and S.-T.~Yau,
\emph{A note on the total curvature of a K\"ahler manifold},
Math. Res. Lett. \textbf{3} (1996), no.~1, 123--132.
\url{https://doi.org/10.4310/MRL.1996.v3.n1.a12}.

\bibitem{WXZVolume}
G.~Wei, G.~Xu, and S.~Zhang,
\emph{Volume growth and positive scalar curvature},
Trans. Amer. Math. Soc. \textbf{378} (2025), no.~9, 6109--6136.
\url{https://doi.org/10.1090/tran/9280}.

\bibitem{XuNonparabolic}
G.~Xu,
\emph{Integral of scalar curvature on non-parabolic manifolds},
J. Geom. Anal. \textbf{30} (2020), no.~1, 901--909.
\url{https://doi.org/10.1007/s12220-019-00174-7}.

\bibitem{Xu}
G.~Xu,
\emph{Integral of scalar curvature on manifolds with a pole},
Proc. Amer. Math. Soc. \textbf{152} (2024), no.~11, 4865--4872.
\url{https://doi.org/10.1090/proc/16584}.

\bibitem{XuCollapse}
G.~Xu,
\emph{Integral of scalar curvature under volume collapse},
preprint, arXiv:2609.10160 (2026).
\url{https://arxiv.org/abs/2609.10160}.

\bibitem{XuSharp2026}
G.~Xu,
\emph{Integral of scalar curvature: Sharp asymptotic bounds and rigidity},
preprint (2026), arXiv:2609.26357v1.
\url{https://arxiv.org/abs/2609.26357v1}.

\bibitem{Yang}
B.~Yang,
\emph{On a problem of Yau regarding a higher dimensional generalization
of the Cohn--Vossen inequality},
Math. Ann. \textbf{355} (2013), no.~2, 765--781.
\url{https://doi.org/10.1007/s00208-012-0803-3}.

\bibitem{YauProblems}
S.-T.~Yau,
\emph{Open problems in geometry},
in \emph{Shiing-Shen Chern---a great geometer of the twentieth century}
(S.-T.~Yau, ed.), International Press, Hong Kong, 1992,
pp.~275--319, Problem~9.

\bibitem{ZhuComparison}
B.~Zhu,
\emph{Comparison theorem and integral of scalar curvature on three manifolds},
J. Geom. Anal. \textbf{32} (2022), no.~7, Paper No.~197, 19~pp.
\url{https://doi.org/10.1007/s12220-022-00934-y}.

\bibitem{ZhuGeometry}
B.~Zhu,
\emph{Geometry of positive scalar curvature on complete manifold},
J. Reine Angew. Math. \textbf{791} (2022), 225--246.
\url{https://doi.org/10.1515/crelle-2022-0049}.

\end{thebibliography}
\end{document}